\documentclass[12pt, reqno]{amsart}
\usepackage{amsmath, amsthm, amscd, amsfonts, amssymb, graphicx, color}
\usepackage[bookmarksnumbered, colorlinks, plainpages]{hyperref}
\hypersetup{colorlinks=true,linkcolor=red, anchorcolor=green, citecolor=cyan, urlcolor=red, filecolor=magenta, pdftoolbar=true}

\newtheorem{theorem}{Theorem}[section]
\newtheorem{lemma}[theorem]{Lemma}

\newtheorem{corollary}[theorem]{Corollary}

\newtheorem{question}[theorem]{Question}
\newtheorem*{TTheorem}{Theorem I (1991)}
\newtheorem*{TBrennan}{Theorem II (2008)}
\newtheorem*{ARSTheorem}{Theorem III (2009)}
\newtheorem*{TTolsa}{Theorem IV (Results of Tolsa)}
\newtheorem*{PFormula}{Plemelj's Formula}
\theoremstyle{definition}

\theoremstyle{remark}
\newtheorem{remark}[theorem]{Remark}
\newtheorem{notation}[theorem]{Notation}
\numberwithin{equation}{section}

\begin{document}

\setcounter{page}{1}

\title[On Nontangential Limits and Shift Invariant subspaces]{On Nontangential Limits and Shift Invariant subspaces}

\author[J. Akeroyd, J. Conway, \MakeLowercase{and} L. Yang]{John R. Akeroyd,$^1$ John B. Conway,$^2$ \MakeLowercase{and} Liming Yang$^3$}

\address{$^{1}$Department of Mathematics, University of Arkansas, Fayetteville, AR 72701}
\email{\textcolor[rgb]{0.00,0.00,0.84}{jakeroyd@uark.edu}}

\address{$^{2}$Department of Mathematics, The George Washington University, Washington, DC 20052}
\email{\textcolor[rgb]{0.00,0.00,0.84}{conway@gwu.edu}}

\address{$^3$Department of Mathematics, Virginia Polytechnic and State University, Blacksburg, VA 24061.}
\email{\textcolor[rgb]{0.00,0.00,0.84}{yliming@vt.edu}}






\begin{abstract} In 1998, John B. Conway and Liming Yang wrote a paper \cite{cy98} in which they posed 
a number of open questions regarding the shift on $P^t(\mu)$ spaces. A few of these have been completely 
resolved, while at least one remains wide open. In this paper, we review some of the solutions, mention 
some alternate approaches and discuss further the problem that remains unsolved.
\end{abstract} \maketitle

\section{\textbf{Introduction}}

Throughout this paper, let $\mathbb{D}$ denote the unit
disk $\{z: |z| < 1\}$ in the complex plane $\mathbb{C}$, let
$\mathbb{T}$ denote the unit circle $\{z: |z| =1\}$, let $A$
denote normalized area measure on $\mathbb{D}$ and let
$m$ denote normalized Lebesgue measure on $\mathbb{T}$. 
Part of our plan here is to review the status of problems posed in 
\cite{cy98}. Some of these have been solved with a vengeance,
while others remain open. We have a brief retrospective on the solutions
in the literature and outline some recent alternatives to these that reach into 
the context of approximation by rational functions. We also make some observations
concerning the problem(s) that remain open. Along the way we mention some related open
questions. We begin by setting the context of our discussion, that includes three major theorems
listed in chronological order. Herein, let $\mu$ be a finite, positive Borel measure that is compactly 
supported in $\mathbb{C}$. Sometimes we require that the support of $\mu$ be contained in
some compact set $K$ and we indicate this by $\mbox{spt}(\mu)~\subseteq~K$. Under these circumstances
and for $1\leq t < \infty$, the analytic polynomials and functions in $\mbox{Rat}(K) := \{q:\mbox{$q$ is a rational function 
with poles off $K$}\}$ are members of $L^t(\mu)$. We let $P^t(\mu)$ denote the closure of the (analytic) 
polynomials in $L^t(\mu)$ and let $R^t(K, \mu)$ denote the closure of $\mbox{Rat}(K)$ in $L^t(\mu)$.
A point $z_0$ in $\mathbb{C}$ is called a \textit{bounded point evaluation} for $P^t(\mu)$ (resp., $R^t(K, \mu)$)
if $f\mapsto f(z_0)$ defines a bounded linear functional for the analytic polynomials (resp., functions in $\mbox{Rat}(K)$)
with respect to the $L^t(\mu)$ norm. The collection of all such points is denoted $\mbox{bpe}(P^t(\mu))$ 
(resp., $\mbox{bpe}(R^t(K, \mu)$)).  If $z_0$ is in the interior of $\mbox{bpe}(P^t(\mu))$ (resp., $\mbox{bpe}(R^t(K, \mu)$)) 
and there exist positive constants $r$ and $M$ such that $|f(z)| \leq M\|f\|_{L^t(\mu)}$, whenever $|z - z_0|\leq r$ 
and $f$ is an analytic polynomial (resp., $f\in \mbox{Rat}(K)$), then we say that $z_0$ is an 
\textit{analytic bounded point evaluation} for $P^t(\mu)$ (resp., $R^t(K, \mu)$). The collection of all such 
points is denoted $\mbox{abpe}(P^t(\mu))$ (resp., $\mbox{abpe}(R^t(K, \mu)$)). Actually, it follows from Thomson's Theorem
\cite{thomson} (or see Theorem I, below) that $\mbox{abpe}(P^t(\mu))$ is the interior of $\mbox{bpe}(P^t(\mu))$. 
This also holds in the context of $R^t(K, \mu)$ as was shown by J. Conway and N. Elias in \cite{ce93}. Now, 
$\mbox{abpe}(P^t(\mu))$ is the largest open subset of $\mathbb{C}$ to which every function in $P^t(\mu)$ has an analytic 
continuation under these point evaluation functionals, and similarly in the context of $R^t(K, \mu)$.

Our story begins with celebrated results of J. Thomson, in \cite{thomson}. 

\begin{TTheorem} 

Let $\mu$ be a finite, positive Borel measure that is compactly supported in $\mathbb{C}$ and suppose that $1\leq t < \infty$.

(a) There is a Borel partition $\{\Delta_i\}_{i=0}^\infty$ of $\mbox{spt}(\mu)$ such that the space $P^t(\mu |_{\Delta_i})$ 
contains no nontrivial characteristic function (i.e., 
 $P^t(\mu |_{\Delta_i})$ is irreducible) and
 \[
 \ P^t(\mu ) = L^t(\mu |_{\Delta_0})\oplus \left \{ \oplus _{i = 1}^\infty P^t(\mu |_{\Delta_i}) \right \}.
 \]
Furthermore, if $U_i :=abpe( P^t(\mu |_{\Delta_i}))$ for $i \ge 1,$ then $U_i$ is a simply connected region and $\Delta_i\subseteq \overline{U_i}$

(b) If $U$ is a bounded, simply connected region, then there is a finite, positive Borel measure that is compactly supported
in $\mathbb{C}$ such that $P^t(\mu )$ is irreducible and $\mbox{abpe}(P^t(\mu)) = U.$

\end{TTheorem}
\bigskip

The next result in our list is due to J. Brennan and it deals with rational approximation, in certain cases.
\bigskip

\begin{TBrennan} Suppose that $1\le t < \infty$ and the diameters of components of $\mathbb C\setminus K$ are bounded below. 
Then there is a Borel partition $\{\Delta_i\}_{i=0}^\infty$ of $\mbox{spt}(\mu)$ such that $R^t(K, \mu |_{\Delta_i})$ is irreducible and
 \[
 \ R^t(K, \mu ) = L^t(\mu |_{\Delta_0})\oplus \left \{ \oplus _{i = 1}^\infty R^t(K, \mu |_{\Delta_i}) \right \}.
 \]
Furthermore, if $U_i :=\mbox{abpe}( R^t(K, \mu |_{\Delta_i}))$, then $\Delta_i\subset \overline{U_i}.$
\end{TBrennan}
\bigskip

Lastly we mention a result of A. Aleman, S. Richter and C. Sunberg. It's proof involves a modification of Thomson's scheme along
with results of X. Tolsa on analytic capacity. 
\bigskip

\begin{ARSTheorem} \label{ARSTheorem}
Suppose that $\mu$ is supported in $\overline{\mathbb D}$, 
$abpe (P^t(\mu )) = \mathbb{D}$, $P^t(\mu )$ is irreducible, and that $\mu (\mathbb{T})> 0$.
\newline
(a) If $f \in P^t(\mu )$, then the nontangential limit $f^*(\zeta )$ of f at $\zeta$ exists a.e. $\mu |_{\mathbb{T}}$ 
and $f^* = f |_{\mathbb{T}}$ as elements of $L^t(\mu |_{\mathbb{T}}).$
\newline
(b) Every nontrivial, closed invariant subspace $\mathcal{M}$ for the shift $S_{\mu}$ on $P^t(\mu )$ has index 1; that is, the dimension
of $\mathcal{M}/z\mathcal{M}$ is one.
\end{ARSTheorem}

\section{\textbf{Boundary Values and Invariant Subspaces}}

We begin this section with a discussion of Theorem III, in Section 1. The hypothesis
of this theorem is that $\mu$ is a finite, positive Borel measure, supported in $\overline{\mathbb{D}}$, $\mu(\mathbb{T}) > 0$,
$\mathbb{D} = \mbox{abpe}(P^t(\mu))$ and $P^t(\mu)$ is irreducible.
It is well-known that these conditions imply that $\mu|_{\mathbb{T}}\ll m$. Part of the strategy of Aleman, 
Richter and Sundberg in establishing this result is centered around showing:
\bigskip

\noindent (\#) If $E$ is a Borel subset of $\mathbb{T}$ and $\mu(E) > 0$, then there is a Jordan subregion $\Omega$ 
of $\mathbb{D}$ such that $\partial{\Omega}$ is rectifiable, $m(E\cap\partial{\Omega}) > 0$ and 
\[f\mapsto f|_{\Omega}\] defines a bounded operator from $P^t(\mu)$ to the Hardy space $H^t(\Omega)$. Moreover,
the nontangential boundary values of $f|_{\Omega}$ on $\mathbb{T}\cap \partial{\Omega}$ coincide with $f$ 
(as a function in $P^t(\mu)$) a.e. $m$. 
\bigskip

\noindent Prior to Theorem~III, it was ``conjectured" that $(\#)$ might be a reachable, intermediate 
target that would lead from the general hypothesis here to the desired conclusion. The precedents in the literature include 
\cite{a01}, where everything was couched in terms of harmonic measure. And, indeed, $(\#)$ was established in 
\cite{a02}, in the special case: $t = 2$ and the support of $\mu$ on $\mathbb{T}$ can be expressed as a union of arcs $\{I_n\}_n$, 
where, for each $n$, 
\[\int_{I_n}\log(\mbox{\small{$\left [\frac{d\mu}{dm}\right ]$}}(\zeta))\, dm(\zeta) > -\infty.\]
The proof given in \cite{a02} makes use of Theorem~I, part (a) (above), but does not lend itself for extension to cases beyond this
special assumption on $\mu|_{\mathbb{T}}$. To establish the general result, Aleman, Richter and Sundberg drove deeply
to rework Theorem~I, part (a), using results of X. Tolsa on analytic capacity and thus gained a much better understanding 
of the norms and locations of bounded point evaluations. In the general setting, they obtained (\#) as a consequence of the hypothesis
that $\mu$ is a finite, positive Borel measure supported in $\overline{\mathbb{D}}$ such that $\mathbb{D} = \mbox{abpe}(P^t(\mu))$,
$P^t(\mu)$ is irreducible and $\mu(\mathbb{T}) > 0$; and they did so with precision (see \cite[Theorem 3.1]{ars02}). 
This, by itself, is a very useful result. With $(\#)$ in hand, one can establish 
nontangential limits quite easily and the index one result is also within reach. Actually, Aleman, Richter and Sundberg establish the 
index one result using just the existence of nontangential limits (cf., \cite[Theorem~3.2]{ars02}). Therefore, Theorem~III answers 
affirmatively Question~2.1 in \cite{cy98}, in its full generality. At the end of this section we give an alternate proof of
Theorem~III, part (b), that depends on $(\#)$. And, in Section~3, we provide an alternate route to the existence of nontangential 
limits that bypasses $(\#)$. All of these results seem to point in the direction of an affirmative answer to \cite[Problem~2.2]{cy98} . 
That is, if the hypothesis of Theorem III is in place and $\mathcal{M}$ is a nontrivial, closed invariant subspace for $S_{\mu}$ on $P^t(\mu)$, then there exists $f$ in $P^t(\mu)$ such that $\mathcal{M} = [f]$ -- the closure of $\{fp:\mbox{$p$ is a polynomial}\}$ 
in $L^t(\mu)$. In other words, the shift on every nontrivial, closed invariant subspace of $P^t(\mu)$ is cyclic. But very little 
progress has been made on this at all, not even in the case that $t=2$. We pose this question once again and follow it with 
some observations.

\begin{question}
Suppose that $\mu$ is supported in $\overline{\mathbb{D}}$, $P^2(\mu)$ is irreducible, $\mbox{abpe}(P^2(\mu)) = \mathbb{D}$ and 
$\mu(\mathbb{T}) > 0$. Is every nontrivial, closed, shift invariant subspace of $P^2(\mu)$ cyclic?
\end{question}
 
Now, in a separate paper (cf., \cite{ars01}), A. Aleman, S. Richter and C. Sundberg show that every nontrivial, closed, shift
invariant subspace $\mathcal{M}$ of the classical Bergman space $\mathbb{A}^2$ is generated by the wandering subspace
$\mathcal{M}\ominus z\mathcal{M}$. So, in character, this follows Beurling's Theorem for the classical Hardy space $H^2$. We should remark 
that S.~Shimorin provided a separate proof of this in \cite{s}. H.~Hedenmalm and K.~Zhu have shown that this does not hold in 
certain weighted Bergman spaces (see \cite{HZ}). In what follows, we modify a special example of Hedenmalm and Zhu to give an 
example of a measure $\mu$ that satisfies the hypothesis of the Question~2.1, but for which there exists a nontrivial, closed shift 
invariant subspace $\mathcal{M}$ of $P^2(\mu)$ such that $\mathcal{M}\ominus z\mathcal{M}$ does not generate 
$\mathcal{M}$. It turns out that this $\mathcal{M}$ is, nevertheless, cyclic; though the vector that generates it is 
not orthogonal to $z\mathcal{M}$.

\begin{theorem}
There exists a finite, positive Borel measure 
$\mu$ with support in $\overline{\mathbb{D}}$ such that:
\begin{itemize}
\item[i)] $\mu(\mathbb{T}) > 0$,
\item[ii)] $P^2(\mu)$ is irreducible,
\item[iii)] Every point in $\mathbb{D}$ is an analytic bounded 
point evaluation for $P^2(\mu)$, and
\item[iv)] There is a closed, shift invariant subspace $\mathcal{M}$
in $P^2(\mu)$ such that $\mathcal{M}\ominus z\mathcal{M}$ does
not generate $\mathcal{M}$.
\end{itemize}
\end{theorem}

\begin{proof}  We turn to \cite{HZ} and consider the case: $\alpha = 5$.  Following the notation of
this paper of Hedenmalm and Zhu, let

\[dA_5 := 6(1-|z|^2)^5dA(z).\] 

\noindent By \cite[Propositions~2 and 3]{HZ}, this measure $A_5$ has all of the properties that 
we are looking for except (i), as stated above. Our strategy is to add some measure to $A_5$ to 
gain condition (i), while preserving (ii) - (iv). Now, by \cite[Proposition~3]{HZ}, there exists $a$ in 
$\mathbb{D}$ such that 

\[g_5(z) := 1 - \left (\frac{1 - |a|^2}{1 - \bar{a}z}\right )^7\]

\noindent has another zero in $\mathbb{D}$ other than $z = a$; and clearly $z = 0$ is not a zero
of $g_5$. Let $\Omega$ be a smoothly bounded Jordan region in $\mathbb{D}$ such that 
\begin{itemize}
\item $0\in \Omega$,
\item $\mathbb{T}\cap \partial{\Omega}$ is a nontrivial subarc of $\mathbb{T}$,
\item $|g_5|$ is bounded below by a positive constant on $\overline{\Omega}$;
indeed, $g_5|_{\Omega}$ is a ``nice" outer function on $\Omega$. 
\end{itemize}
Let $\omega_{\Omega}$ denote harmonic measure on $\partial{\Omega}$
for evaluation at $z = 0$ and define $\sigma$ on $\partial{\Omega}$ by

\[d\sigma = \frac{1}{|g_5|^2}d\omega_{\Omega}.\]

\noindent Then, for any polynomial $p$ with $p(0) = 0$, we have:

\[\int_{\partial{\Omega}}p(\zeta)g_5(\zeta)\cdot\overline{g_5(\zeta)}\, d\sigma(\zeta) = 0.\]

\noindent Since $g_5|_{\Omega}$ is an outer function in $\Omega$, we find that 
$\{pg_5: \mbox{$p$ is a polynomial and $p(0) = 0$}\}$ is dense in $H_0^2(\Omega) : = \{f\in H^2(\Omega): f(0) = 0\}$, 
and indeed is dense in\linebreak $\{f~\in~P^2(\sigma)~:~f(0)~=~0\}$. From this it follows that 
\ \begin{eqnarray}\label{eq3-1}
\ \int_{\partial{\Omega}}f(\zeta)\cdot\overline{g_5(\zeta)}\, d\sigma(\zeta) = 0,
\ \end{eqnarray}
whenever $f\in P^2(\sigma)$ and $f(0) = 0$.
Define $\mu$ on $\overline{\mathbb{D}}$ by: $\mu = \sigma + A_5$.
By standard arguments, $\mu$ satisfies (i)-(iii), as listed above.
Moreover, since $A_5\leq \mu$ and $\sigma \leq \mu$ it follows that if $f\in P^2(\mu)$, then $f|_{\mathbb{D}}\in P^2(A_5)$
and $f|_{\overline{\Omega}}\in P^2(\sigma)$. Notice that $g_5\in H^{\infty}(\mathbb{D})$ and hence $g_5\in P^2(\mu)$. 
Let $\mathcal{M} = \{f\in P^2(\mu): f(a) = 0\}$. Then, by the \eqref{eq3-1} and \cite[Proposition~2]{HZ}, $g_5\in \mathcal{M}$
and $g_5\perp z\mathcal{M}$. Now, by \cite{a02} or Theorem~III above, $\mathcal{M}\ominus z\mathcal{M}$ 
has dimension one, and so $\mathcal{M}\ominus z\mathcal{M} = \{cg_5: c\in\mathbb{C}\}$.
However, $\mathcal{M}\neq [g_5]$, since $g_5$ has another zero in $\mathbb{D}$ other than $z=a$.
It is a straightforward exercise to show that $\mathcal{M} = [h]$, where $h(z) := z-a$.
\end{proof}

We close this section by developing an alternate proof of Theorem III, part (b). 
Our motivation for including it is that anything in this direction could be useful in establishing Question 2.1, above.
We are assuming that (\#) holds as stated in the beginning of this section. This, coupled with work in \cite{a01}, in particular,
\cite[Theorem 2.4]{a01} -- which easily carries over for $1\leq t < \infty$ -- and \cite[Theorem 3.2]{a01}, allows us to reduce to the case:
\bigskip

\noindent\textbf{Reduction.} There are Jordan subregions $V$ and $W$ of $\mathbb{D}$ such that:
\begin{itemize}
\item $0\in V\subseteq W$ and $\partial{V}$ is rectifiable. 
\item $\mu(W) = 0$ and hence $\mu(V) = 0$.
\item $\omega_V(\mathbb{T}) > 0$ and hence $\omega_W(\mathbb{T}) > 0$; where $\omega_V$ is harmonic measure on 
$\partial{V}$ for evaluation at $0$, and similarly for $\omega_W$. 
\item $g\mapsto g|_W$ defines a bounded operator from $P^t(\mu)$ to $H^t(W)$ -- the Hardy space of $W$ -- and so the same 
holds with $V$ in place of $W$. 
\item If $g\in L^1(\mu)$ and $\int pg\, d\mu = 0$ for all polynomials $p$, then  
$\hat{g}(\zeta) := \int\frac{g(z)}{z-\zeta}\, d\mu(z)$ is in $H^1(V)$.
\item Again, we can express $d\mu|_{\mathbb{T}}$ as $hdm$.
\end{itemize}
\bigskip

We recall that $\hat{g}$, as defined above on $V$, is called the Cauchy transform of the measure $gd\mu$ (on $V$).
In Section 3 we shall have a much more careful discussion of Cauchy transforms and we shall use the alternate 
notation of $\mathcal C(g\mu)$. Our next lemma is quite standard in the literature. We include a proof for the sake of completeness.

\begin{lemma} Assume that $\sigma\in L^1(\mu)$ and that
\[\int p(z)z\sigma(z)\, d\mu(z) = 0,\]
for all polynomials $p$. That is, $z\sigma\perp \mathcal{P}$ (the collection of analytic polynomials). 
Then we have: $(z\sigma\hat{)\,}\in H^1(V)$ and the nontangential boundary values of 
$(z\sigma\hat{)\,}$ on $\mathbb{T}\cap \partial{V}$ are $\sigma h$, a.e. $m$.
\end{lemma}

\begin{proof} It is a consequence of \cite[Theorem 3.2]{a01} that $(z\sigma\hat{)\,}\in H^1(V)$ -- see our
``Reduction" assumption here. So we really only need to establish the nontangential 
boundary values assertion of this lemma. Now, since $\partial{V}$ is rectifiable, $\partial{V}$ is tangent to $\mathbb{T}$
a.e. $[m]$ on $\mathbb{T}\cap \partial{V}$. Let $\Lambda = \{\zeta\in\mathbb{T}: \mbox{$\partial{V}$ is 
tangent to $\mathbb{T}$ at $\zeta$}\}$. Then, for any $\zeta$ in $\Lambda$, there exists $\lambda$, 
$0 < \lambda < 1$, such that $\{r\zeta: \lambda\leq r < 1\}\subseteq V$. Since $\int p(z)z\sigma(z)\, d\mu(z) = 0$
for all polynomials $p$, if $\lambda \leq r < 1$, then we have:
\begin{eqnarray*}
(z\sigma\hat{)\,}(r\zeta) &=& \int\frac{z\sigma(z)}{z - r\zeta}\, d\mu(z)\\
&=&  \int\left (\frac{1}{z - r\zeta} - \frac{1}{z - \frac{1}{r\overline{\zeta}}}\right )z\sigma(z)\, d\mu(z)\\
&=& \int\left (\frac{1}{z - r\zeta} + \frac{r\overline{\zeta}}{1- r\overline{\zeta}z}\right )z\sigma(z)\, d\mu(z)\\
&=& \int \frac{1- r^2}{(z - r\zeta)(1- r\overline{\zeta}z)}z\sigma(z)\, d\mu(z)\\
&=& \int \frac{1- r^2}{|1- r\overline{\zeta}z|^2}\left (\frac{1 - r\zeta\overline{z}}{z - r\zeta}\right )z\sigma(z)\, d\mu(z)\\
&=& \int \frac{1- r^2}{|1- r\overline{\zeta}z|^2}\left (\frac{z - r\zeta\ |z|^2}{z - r\zeta}\right )\sigma(z)\, d\mu(z)\\
&=& \int \frac{1- r^2}{|1- r\overline{\zeta}z|^2}\sigma(z)\, d\mu(z) + \\
&&\hspace*{2in}\int \frac{1- r^2}{|1- r\overline{\zeta}z|^2}\left (\frac{r\zeta (1 - |z|^2)}{z - r\zeta}\right )\sigma(z)\, d\mu(z)
\end{eqnarray*}
Now, 
\[\int \frac{1- r^2}{|1- r\overline{\zeta}z|^2}\sigma(z)\, d\mu(z) = \int_{\mathbb{D}} \frac{1- r^2}{|1- r\overline{\zeta}z|^2}\sigma(z)\, d\mu(z) 
+ \int_{\mathbb{T}} \frac{1- r^2}{|1- r\overline{\zeta}z|^2}\sigma(z)\, h(z)dm(z).\]
It is well-known that 
\[\int_{\mathbb{T}} \frac{1- r^2}{|1- r\overline{\zeta}z|^2}\sigma(z)\, h(z)dm(z)\longrightarrow \sigma(\zeta)h(\zeta),\]
as $r\rightarrow 1^-$, for $m$ - a.a. $\zeta$ in $\mathbb{T}$. And, by a result of T. Kriete and T. Trent (cf., \cite[Lemma~1.1]{ars02}), 
\[\int_{\mathbb{D}} \frac{1- r^2}{|1- r\overline{\zeta}z|^2}\sigma(z)\, d\mu(z)\longrightarrow 0,\]
as $r\rightarrow 1^-$, for $m$ - a.a. $\zeta$ in $\mathbb{T}$. Therefore,  
\[\int \frac{1- r^2}{|1- r\overline{\zeta}z|^2}\sigma(z)\, d\mu(z) \longrightarrow \sigma(\zeta)h(\zeta),\]
as $r\rightarrow 1^-$, for $m$ - a.a. $\zeta$. Furthermore, 
\[\int \frac{1- r^2}{|1- r\overline{\zeta}z|^2}\left (\frac{r\zeta (1 - |z|^2)}{z - r\zeta}\right )\sigma(z)\, d\mu(z) = 
\int_{\mathbb{D}} \frac{1- r^2}{|1- r\overline{\zeta}z|^2}\left (\frac{r\zeta (1 - |z|^2)}{z - r\zeta}\right )\sigma(z)\, d\mu(z).\]
Applying \cite[Lemma 1.1]{ars02} once again, we have: 
\[\int_{\mathbb{D}} \frac{1- r^2}{|1- r\overline{\zeta}z|^2}|\sigma(z)|\, d\mu(z)\longrightarrow 0,\]
as $r\rightarrow 1^-$, for $m$ a.a. $\zeta$ in $\mathbb{T}$. If $\zeta$ is a point of tangency of $\partial{V}$
with $\mathbb{T}$, then there is a positive constant $M$ (depending on $\zeta$) such that, for all $r$ sufficiently near $1$, 
\[\left |\frac{r\zeta (1 - |z|^2)}{z - r\zeta}\right |\leq M,\]
for all $z\in \mathbb{D}\setminus{W}$. It now follows that 
\[\int_{\mathbb{D}} \frac{1- r^2}{|1- r\overline{\zeta}z|^2}\left (\frac{r\zeta (1 - |z|^2)}{z - r\zeta}\right )\sigma(z)\, d\mu(z)\longrightarrow 0,\]
as $r\rightarrow 1^-$, for $m$ - a.a. points in $\Lambda$. Putting all of this together, we have:
\[(z\sigma\hat{)\,}(r\zeta) \longrightarrow \sigma(\zeta)h(\zeta),\]
as $r\rightarrow 1^-$, for $m$ - a.a. points $\zeta$ in $\Lambda$; and our proof is complete.
\end{proof}

\noindent We now give an alternate proof of the index one result: Theorem III, part (b). Since our ``Reduction"
is still in place, we are assuming (\#) here, as given above. This proof has things in common with the 
proofs of \cite[Theorem 1]{ot} and \cite[Theorem 3.2]{ars02}, though it is different in some respects from either.
\bigskip

\begin{theorem} If $\mathcal{M}$ is a closed subspace of $P^t(\mu)$ that is shift invariant, 
then dim$(\mathcal{M}/z\mathcal{M}) = 1$.
\end{theorem}

\begin{proof}
We argue indirectly and suppose that there is a closed, shift invariant subspace $\mathcal{M}$ of $P^t(\mu)$ such that 
$\mbox{dim}(\mathcal{M}/z\mathcal{M}) \geq 2$. Then, we can find $f$ and $g$ in $\mathcal{M}\setminus z\mathcal{M}$ such 
that $f\not\in\{cg: c\in\mathbb{C}\} + z\mathcal{M}$ and $g\not\in\{cf: c\in\mathbb{C}\} + z\mathcal{M}$.
By the first paragraph of the proof of Theorem 3.2 in \cite{ars02}, we may assume that $f(0) \neq 0$. Now, since 
$f\not\in\{cg: c\in\mathbb{C}\} + z\mathcal{M}$ and $g\not\in\{cf: c\in\mathbb{C}\} + z\mathcal{M}$, 
there exists $\beta_1$ and $\beta_2$  in $L^s(\mu)$ ($\frac{1}{s} + \frac{1}{t} = 1$) such that:
\begin{itemize}
\item $\beta_1\perp \{cg: c\in\mathbb{C}\} + z\mathcal{M}$ (i.e., $\int q(z)\beta_1(z)\, d\mu(z) = 0$ for all $q$ in 
$\{cg: c\in\mathbb{C}\} + z\mathcal{M}$), yet $\beta_1\not\perp f$, and 
\item $\beta_2\perp \{cf: c\in\mathbb{C}\} + z\mathcal{M}$, yet $\beta_2\not\perp g$. 
\end{itemize}
Let $c_1 = \int f\beta_1\, d\mu$ and let $c_2 = \int g\beta_2\, d\mu$. Notice that we can choose $c_1$ and $c_2$ to 
be any nonzero constants we wish. Let $\beta = \beta_1 + \beta_2$.
With $\beta$ as defined above:
\begin{itemize}
\item $zf\beta\perp \mathcal{P}$ (the collection of analytic polynomials) and $zg\beta\perp \mathcal{P}$, 
\item $c_1 = \int f\beta\, d\mu = (zf\beta\hat{)\,}(0)$, and
\item $c_2 = \int g\beta \, d\mu = (zg\beta\hat{)\,}(0)$.
\end{itemize}

\noindent Define $\varphi$ and $\psi$ on $V$ by:
\begin{itemize}
\item $\varphi(\zeta) = f(\zeta)(zg\beta\hat{)\,}(\zeta)$, and
\item $\psi(\zeta) = g(\zeta)(zf\beta\hat{)\,}(\zeta)$.
\end{itemize}

\noindent Now, by our ``Reduction" and Lemma 2.3, $\varphi$ and $\psi$ are both in $\mathcal{N}(V)$ (the Nevanlinna class of $V$) 
and they have the same nontangential boundary values a.e. $\omega_{V}$ on $\mathbb{T}\cap \partial{V}$ (which has positive 
$\omega_V$-measure). Moreover, $\varphi(0) = c_1f(0)$, and $\psi(0) = c_2g(0)$. Since $f(0)\neq 0$ and our choice of $c_1$ and $c_2$ 
among the nonzero constants can be specified without restraint, we can force: $\varphi(0)\neq \psi(0)$. So, we can force: 
$\varphi - \psi\in\mathcal{N}(V)$, $\varphi - \psi$ has zero nontangential boundary values a.e. $\omega_{V}$ on 
$\mathbb{T}\cap \partial{V}$ (which has positive $\omega_V$-measure), yet $(\varphi - \psi)(0) \neq 0$. But, this cannot happen for 
Nevanlinna class functions and so we have a contradiction.
\end{proof}
\vspace*{.5in}

\section{\textbf{Boundary Values, Another Way}}

J. Thomson's proof of the existence of bounded point evaluations for $P^t(\mu)$
uses Davie's deep estimation of analytic capacity, S. Brown's technique, and Vitushkin's localization for 
uniform rational approximation. The proof is excellent but complicated, and it does not really
lend itself to showing the existence of nontangential boundary values in the case 
that $\mbox{spt}(\mu)\subseteq\overline{\mathbb{D}}$, $P^t(\mu)$ is irreducible and $\mu(\mathbb{T}) > 0$.
X. Tolsa's remarkable results on analytic capacity opened the door for a new view of things, through the works
of \cite{ars02}, \cite{ARS10} and \cite{b08}. In this section we present an alternate route to boundary values 
that has extension to the context of mean rational approximation. It also uses the results of X. Tolsa on analytic capacity.
For $\lambda$ in $\mathbb{C}$ and $r > 0$, we let $B(\lambda, r) = \{z\in\mathbb{C}: |z - \lambda | < r\}$.

Let $\nu$ be a finite complex-valued Borel measure that
is compactly supported in $\mathbb {C}$. 
For $\epsilon > 0,$ $\mathcal C_\epsilon(\nu)$ is defined by
\ \begin{eqnarray}
\ \mathcal C_\epsilon(\nu)(z) = \int _{|w-z| > \epsilon}\dfrac{1}{w - z} d\nu (w).
\ \end{eqnarray} 
The (principal value) Cauchy transform
of $\nu$ is defined by
\ \begin{eqnarray}\label{CTDefinition}
\ \mathcal C(\nu)(z) = \lim_{\epsilon \rightarrow 0} \mathcal C_\epsilon(\nu)(z)
\ \end{eqnarray}
for all $z\in\mathbb{C}$ for which the limit exists.
If $\lambda \in \mathbb{C}$ and $\int \frac{d|\nu |}{|z - \lambda |} < \infty$, then 
$\lim_{r\rightarrow 0}\frac{|\nu |(B(\lambda, r))}{r} = 0$ and 
$\lim_{\epsilon \rightarrow 0} \mathcal C_{\epsilon}(\nu)(\lambda )$ exists. Therefore, a standard application of Fubini's
Theorem shows that $\mathcal C(\nu) \in L^s_{\mbox{loc}}(\mathbb{C})$, for $ 0 < s < 2$. In particular, it is
defined for almost all $z$ with respect to area measure on $\mathbb{\mathbb{C}}$, and clearly $\mathcal{C}(\nu)$ is analytic 
in $\mathbb{C}_\infty \setminus\mbox{spt}(\nu)$, where $\mathbb{C}_\infty := \mathbb{C} \cup \{\infty \}.$ In fact, from 
Corollary \ref{ZeroAC} below, we see that \eqref{CTDefinition} is defined for all $z$ except for a set of zero analytic 
capacity. Thoughout this section, the Cauchy transform of a measure always means the principal value of the transform.

The maximal Cauchy transform is defined by
 \[
 \ \mathcal C_*(\nu)(z) = \sup _{\epsilon > 0}| \mathcal C_\epsilon(\nu)(z) |.
 \]

If $K \stackrel{\mbox{\tiny{$\subset\subset$}}}{} \mathbb{C}$ (i.e.,  $K$ is a compact subset of $\mathbb{C}$), then  we
define the analytic capacity of $K$ by
\[
\ \gamma(K) = \sup |f'(\infty)|,
\]
where the supremum is taken over all those functions $f$ that are analytic in $\mathbb C_{\infty} \setminus K$ such that
$|f(z)| \le 1$ for all $z \in \mathbb{C}_\infty \setminus K$; and
$f'(\infty) := \lim _{z \rightarrow \infty} z[f(z) - f(\infty)].$
The analytic capacity of a general subset $E$ of $\mathbb{C}$ is given by: 
\[
\ \gamma (E) = \sup \{\gamma (K) : K\stackrel{\mbox{\tiny{$\subset\subset$}}}{} E\}.
\]
Good sources for basic information about analytic
capacity are Chapter VIII of \cite{gamelin}, Chapter V of \cite{conway}, and \cite{Tol14}.

A related capacity, $\gamma _+,$ is defined for subsets $E$ of $\mathbb{C}$ by:
\[
\ \gamma_+(E) = \sup \|\mu \|,
\]
where the supremum is taken over positive measures $\mu$ with compact support
contained in $E$ for which $\|\mathcal{C}(\mu) \|_{L^\infty (\mathbb{C})} \le 1.$ 
Since $\mathcal C\mu$ is analytic in $\mathbb{C}_\infty \setminus \mbox{spt}(\mu)$ and $(\mathcal{C}(\mu)'(\infty) = \|\mu \|$, 
we have:
\[
\ \gamma _+(E) \le \gamma (E)
\]
for all subsets $E$ of $\mathbb{C}$.  X. Tolsa has established the following astounding results.

\begin{TTolsa} \label{Results of Tolsa}

(1) $\gamma_+$ and $\gamma$ are actually equivalent. 
That is, there is an absolute constant $A_T$ such that 
\begin{eqnarray}\label{GammaEq}
\ \gamma (E) \le A_ T \gamma_+(E)
\end{eqnarray}
for all $E \subset \mathbb{C}.$ 

(2) Semiadditivity of analytic capacity:
\begin{eqnarray}\label{Semiadditive}
\ \gamma \left (\bigcup_{i = 1}^m E_i \right ) \le A_T \sum_{i=1}^m \gamma(E_i)
\end{eqnarray}
where $E_1,E_2,...,E_m \subset \mathbb{C}.$

(3) There is an absolute positive constant $C_T$ such that, for any $a > 0$, we have:  
\begin{eqnarray}\label{WeakOneOne}
\ \gamma(\{\mathcal{C}_*(\nu)  \geq a\}) \le \dfrac{C_T}{a} \|\nu \|.
\end{eqnarray}

(4) Let $\mu$ be a finite, positive Borel measure on $\mathbb{C}$ with linear growth such that the 
Cauchy transform is bounded in $L^2(\mu )$. Then, for any finite complex-valued Borel measure 
$\nu$ with compact support in $\mathbb{C}$, $\lim_{\epsilon\rightarrow 0}\mathcal{C}_{\epsilon}(\nu)(z)$ exists 
for $\mu - a.e.$ $z$ in $\mathbb{C}$.
\end{TTolsa}

\begin{proof}
(1) and (2) are from \cite{Tol03} (also see Theorem 6.1 and Corollary 6.3 in \cite{Tol14}).

(3) follows from Proposition 2.1 of \cite{Tol02} (also see \cite{Tol14} Proposition 4.16).

For (4), see \cite{Tol03} (also Theorem 8.1 in \cite{Tol14}).
\end{proof} 

\begin{corollary}\label{ZeroAC}
Suppose that $\nu$ is a finite, complex-valued Borel measure with compact support in $\mathbb{C}$. Then there exists 
$E \subset \mathbb{C}$ with $\gamma(E) = 0$ such that $\lim_{\epsilon \rightarrow 0}\mathcal{C} _{\epsilon}(\nu)(z)$ 
exists for $z\in\mathbb{C}\setminus E$.
\end{corollary}

\begin{proof}
Let $F\subset \mathbb{C}$ be a compact subset such that $\lim_{\epsilon \rightarrow 0}\mathcal{C}_{\epsilon}(\nu)(z)$ 
does not exist for all $z\in F.$ Then by Tolsa's Theorem (1) and Theorem 4.14 in \cite{Tol14}, there exist an absolute 
constant $c > 0$ and a positive finite Borel measure $\mu$ with support in $F$ such that $\mu$ is linear growth, 
$\mathcal{C}(\mu)$ is bounded in $L^2(\mu ),$ and 
 \[ 
 \ c \gamma(F) \le \mu(F).
 \] 
By Tolsa's Theorem (4), we conclude that $\mu (F) = 0.$ Hence, $\gamma(F) = 0.$ 
\end{proof}

\begin{lemma}\label{CauchyTLemma} 
Let $\nu$ be a finite, complex-valued Borel measure that is compactly
supported in $\mathbb{C}$ and assume that for some $\lambda _0$ in $\mathbb C$ we have:
\begin{itemize}
\item[(a)] $\lim_{r\rightarrow 0} \dfrac{| \nu |(B(\lambda _0, r))}{r } = 0$ and 
\item[(b)] $\lim_{\epsilon \rightarrow 0}\mathcal{C} _{\epsilon}(\nu)(\lambda _0)$ exists.
\end{itemize}
Then, for any $a > 0$, there exists $\delta_a$, $0 < \delta_a < \frac{1}{4}$, such that whenever $0 < \delta < \delta_a$,  
there is a subset $E_{\delta}$ of $\overline{B(\lambda _0, \delta)}$ 
and $\epsilon (\delta ) > 0$ satisfying: 
 \ \begin{eqnarray}\label{CT1}
 \ \lim _{\delta \rightarrow 0} \epsilon(\delta ) = 0,
 \ \end{eqnarray} 
 \ \begin{eqnarray}\label{CT2}
 \ \gamma(E_\delta ) <\epsilon (\delta ) \delta ,
 \ \end{eqnarray}
and for all $\lambda\in B (\lambda _0, \delta ) \setminus E_\delta,$ $\lim_{\epsilon \rightarrow 0}\mathcal{C} _{\epsilon}(\nu)(\lambda )$ 
exists and 
\ \begin{eqnarray}\label{CT3}
\ |\mathcal{C}(\nu)(\lambda) - \mathcal{C}(\nu)(\lambda _0) | \le a.
\ \end{eqnarray}
\end{lemma}

\begin{proof} Let $M = \sup _{r > 0 }\frac{| \nu |(B(\lambda _0, r))}{r }$. Then, by (a), $M < \infty.$
For $a > 0$, choose $N$ and $\delta_a$, $0 < \delta_a < \frac{1}{4}$, satisfying:
 \[
 \ N = \dfrac{30M}{a} + 2,
 \]
 \[
 \ \dfrac{| \nu |(B(\lambda _0, N\delta))}{\delta } < \dfrac{a}{6},
 \]
and
 \[
 \ |\mathcal{C}_{\delta}(\nu)(\lambda _0) - \mathcal{C}(\nu)(\lambda _0)  | \le \dfrac{a}{6}
 \]
for $0 < \delta < \delta_a.$
We now fix $\delta$, $0 < \delta < \delta_a$, and let $\nu_{N\delta} = \chi _{B (\lambda _0, N\delta )} \nu $, where $\chi _A$ denotes the characteristic function of the set $A$.  
For $0 < \epsilon < \delta $ and $\lambda$ in $B(\lambda _0, \delta )$, we have:
 \[
 \ B(\lambda , \epsilon) \subseteq B(\lambda _0, N\delta )
 \]
and
 \[ 
 \ \begin{aligned}
 \ & |\mathcal{C}_{\epsilon}(\nu)(\lambda) - \mathcal{C}(\nu)(\lambda _0)| \\
 \ \le & |\mathcal{C}_{\epsilon}(\nu)(\lambda) - \mathcal{C}_{\delta}(\nu)(\lambda _0)| + \dfrac{a}{6} \\
 \ \le & |\mathcal{C}_{\epsilon}(\nu - \nu_{N\delta})(\lambda) - \mathcal{C}_{\delta}(\nu - \nu_{N\delta}) (\lambda _0)| + 
|\mathcal{C}_{\epsilon}(\nu_{N\delta})(\lambda)| + |\mathcal{C}_{\delta}(\nu_{N\delta})(\lambda _0)| + \dfrac{a}{6} \\
 \ \le & \left | \int _{\mathbb{C}\setminus B(\lambda _0, N\delta)}\dfrac{(\lambda - \lambda _0)d\nu}{(z - \lambda)(z - \lambda_0)} \right | +  
\mathcal C_*(\nu_{N\delta})(\lambda) + \dfrac{| \nu |(B(\lambda _0, N\delta))}{\delta } + \dfrac{a}{6}\\
 \ \le & \delta \sum_{k = 0}^\infty  \int _{2^{k}N\delta \le |z - \lambda _0| < 2^{k + 1}N\delta}\dfrac{d |\nu |}{|z - \lambda || z - \lambda_0|} +  
\mathcal C_*(\nu_{N\delta})(\lambda) + \dfrac{a}{3}\\
\ \le & \delta \sum_{k = 0}^\infty \dfrac{|\nu |(B (\lambda _0, 2^{k + 1}N\delta ))}{2^{k}(N-1)\delta2^{k}N\delta} +  
\mathcal C_*(\nu_{N\delta})(\lambda) + \dfrac{a}{3}\\
 \ \le & \dfrac{4M}{N-1}  +  \mathcal{C}_*(\nu_{N\delta})(\lambda) + \dfrac{a}{3}\\
 \ \le & \dfrac{a}{2}  +  \mathcal{C}_* (\nu_{N\delta})(\lambda).
 \ \end{aligned}
 \]
Let

 \[
 \ \mathcal{E}_{\delta} = \{\lambda : C_*(\nu _{N\delta})(\lambda ) \ge \frac{a}{2} \} \cap \overline{B (\lambda _0, \delta)}.
 \]
Then
 \[
 \ \{\lambda : |\mathcal C_{\epsilon}(\nu)(\lambda) - \mathcal{C}(\nu)(\lambda _0)| \ge a \} 
 \cap \overline{B (\lambda _0, \delta)} \subset \mathcal{E}_{\delta}.
 \]
From Tolsa's Theorem (3), we get
 \[
 \ \gamma (\mathcal{E}_{\delta}) \le \dfrac{2C_T}{a} \| \nu _{N\delta} \| \le \dfrac{2C_T\delta }{a}\dfrac{| \nu |(B(\lambda _0, N\delta))}{\delta } .
 \]
Let $E$ be the set of $\lambda\in \mathbb{C}$ such that $\mathcal{C}(\nu)(\lambda)$ does not exist. By Corollary \ref{ZeroAC}, 
we see that $\gamma (E) = 0$. Now let $E_{\delta} = \mathcal{E}_{\delta} \cup E$. Applying Tolsa's Theorem (2) we find that
 \[
 \ \gamma(E_{\delta}) \le A_T (\gamma (\mathcal{E}_{\delta})+ \gamma (E)) < \dfrac{2A_TC_T}{a}\dfrac{| \nu |(B(\lambda _0, N\delta))}{\delta }  \delta. 
 \]
Letting
 \[
 \ \epsilon (\delta) = \dfrac{2A_TC_T}{a}\dfrac{| \nu |(B(\lambda _0, N\delta))}{\delta }, 
 \]
we find that \eqref{CT1} and \eqref{CT2} hold.
On $B (\lambda _0, \delta ) \setminus E_\delta$ and for $\epsilon < \delta$, we conclude that
 \[
 \ |\mathcal{C}_{\epsilon}(\nu)(\lambda) - \mathcal{C}(\nu)(\lambda _0)| < a.
 \]
Therefore, \eqref{CT3} follows since
 \[
 \ \lim_{_\epsilon\rightarrow 0} \mathcal{C}_{\epsilon}(\nu)(\lambda) = \mathcal{C}(\nu)(\lambda).
 \]
\end{proof}

\begin{remark}
(1) The above lemma, which is needed in several places of this paper, is a generalization of Lemma 4 in \cite{y17}.

(2) By Lemma 1 of \cite{b67}, we see that $|\mathcal C\nu (\lambda) - \mathcal C\nu (\lambda_0)| \le a$ on a set having 
full area density at $\lambda_0$ whenever $|\lambda - \lambda_0 |$
is sufficiently small. Lemma \ref{CauchyTLemma} shows that this inequality \eqref{CT3} holds for capacitary density which 
is needed (not area density as Browder considers) in our situations.
\end{remark}

\begin{notation}
Part of the statement of Lemma 3.2 that culminates in inequality (3.8) tells us that the sets $E_{\delta}$ can be chosen so that
\[\sup\{|\mathcal{C}(\nu)(\lambda) - \mathcal{C}(\nu)(\lambda _0)|: \lambda\in B (\lambda _0, \delta ) \setminus E_\delta\}\rightarrow 0,\]
as $\delta\rightarrow 0$. From this point on, let us adopt the notation ``$\stackrel{\delta}{\approx}$" to indicate such a phenomenon, which,
in this particular case would read:
\[  \underset{\lambda \in B(\lambda _0, \delta) \setminus E_ {\delta} }{ \mathcal{C}(\nu)(\lambda)} \stackrel{\delta}{\approx} \mathcal{C}(\nu)(\lambda_0).\] 
Furthermore, for $0 < r < 1$ and $\zeta$ in $\mathbb{T}$, let $S_r(\zeta)$ denote the interior of the closed,
convex hull of $\{z: |z|\leq r\}\cup\{\zeta\}$, let $\ell_{\zeta}$ denote the line in $\mathbb{C}$
that is tangent to $\mathbb{T}$ at $\zeta$ and let $T_r(\zeta)$ denote the reflection of $S_r(\zeta)$
through $\ell_{\zeta}$. For $0 < \delta < 1$, let $S_r(\zeta, \delta) = S_r(\zeta)\cap B(\zeta, \delta)$ 
and let $T_r(\zeta, \delta) = T_r(\zeta)\cap B(\zeta, \delta)$ . As before, let $m$ denote normalized 
Lebesgue measure on $\mathbb{T}$.
\end{notation}

\begin{PFormula}\label{PlemeljFormula}
(The classical version for the unit circle $\mathbb{T}$) Let $d\nu = hdm$, where $h\in L^1(m)$. 
Then there exists a subset $Z$ of $\mathbb{T}$, with $m(Z) = 0$, such that for $\zeta$ in $\mathbb{T}\setminus Z$ 
the following hold.

(a) $\mathcal{C}(\nu ) (\zeta) = \lim_{\epsilon\rightarrow 0} \mathcal{C}_ {\epsilon}(\nu)(\zeta)$ exists,

(b)
 \[
 \ \underset{\lambda \in S_r(\zeta)}{\lim _{\lambda\rightarrow \zeta}}\mathcal{C}(\nu )(\lambda) = 
\mathcal{C}(\nu)(\zeta) + \frac{1}{2}h(\zeta)\bar{\zeta},
 \]
and

(c)
\[
 \ \underset{\lambda \in T_r(\zeta)}{\lim _{\lambda\rightarrow \zeta}}\mathcal{C}(\nu )(\lambda) = 
\mathcal{C}(\nu)(\zeta) - \frac{1}{2}h(\zeta)\bar{\zeta},
 \]
\end{PFormula}

Now, $\mbox{abpe}(P^2(m) = \mathbb{D}$. Indeed, $P^2(m)$ is the classical Hardy space $H^2$ and hence every 
function in $P^2(m)$ has a natural analytic continuation to $\mathbb{D}$.  We now apply the classical Plemelj's formula 
(above) to show that every $f$ in $P^2(m)$ has nontangential limits. To this end, suppose $g\perp P^2(m)$. By the proof
of \cite[Chapter VII, Lemma 1.7]{jbc}, we may assume that $|g| > 0$ a.e. $m$. Then, for $\lambda$ in $\mathbb{D}$,
 \[
 \ f(\lambda) \mathcal{C}(gm)(\lambda) = \mathcal{C}(fgm)(\lambda).
 \]
Using Plemelj's formula (b) above, we get
\begin{eqnarray}\label{limitEq}
 \ \underset{\lambda\in S_r(\zeta)}{\lim _{\lambda\rightarrow \zeta}} f(\lambda)\mathcal C(gm)(\lambda)  = 
\mathcal{C}(fgm) (\zeta) + \frac{1}{2}f(\zeta)g(\zeta)\bar{\zeta}, 
 \end{eqnarray} 
for $m$ a.a. (almost all) $\zeta$ in $\mathbb{T}$.
Now, if $\lambda\in\mathbb{C}\setminus\overline{\mathbb{D}}$, then $\frac{1}{z-\lambda},~ \frac{f}{z-\lambda}\in P^2(m)$, and so:
 \[
 \ \mathcal C(gm) (\lambda ) = \mathcal C(fgm) (\lambda ) = 0.
 \]
Therefore, applying Plemelj's formula (c) above, we have
 \[
 \ \ \mathcal{C}(gm)(\zeta) = \frac{1}{2}g(\zeta)\bar{\zeta},
 \]
for $m$ a.a. $\zeta$ in $\mathbb{T}$. And,
 \[
 \ \ \mathcal C(fgm) (\zeta) = \frac{1}{2}f(\zeta) g(\zeta)\bar{\zeta},
 \]
for $m$ a.a. $\zeta$ in $\mathbb{T}$.
Together with \eqref{limitEq}, we find that
 \[
 \ \underset{\lambda \in S_r(\zeta)}{\lim _{\lambda\rightarrow \zeta}} f(\lambda ) g(\zeta)\bar{\zeta}  = f(\zeta) g(\zeta)\overline{\zeta}
 \]
for $m$ a.a. $\zeta$ in $\mathbb{T}$. Since $g(\zeta) \ne 0$, for $m$ a.a. $\zeta$ in $\mathbb{T}$, we conclude that 
\[
 \ \underset{\lambda \in S_r(\zeta)}{\lim _{\lambda\rightarrow \zeta}} f(\lambda)  = f(\zeta),
\]
for $m$ a.a. $\zeta$ in $\mathbb{T}$.

We now develop a generalized Plemelj's Formula that has application to a broad range of $P^t(\mu)$ and $R^t(K, \mu)$
spaces. It is known that Plemelj's Formula holds for some rectifiable curves such as Lipschitz 
graphs (see Theorem 8.8 in \cite{Tol14}). However, for simplicity, our focus is the case of the unit circle $\mathbb{T}$. 
We first need a lemma. 

\begin{lemma}\label{lemmaBasic}
Let $\nu$ be a finite, complex-valued Borel measure with compact support in $\mathbb{C}$. Suppose that $\nu$ is singular to $m$ 
(i.e., $|\nu |\perp m$). Then
 \begin{eqnarray}
 \ m(\{\lambda  :\underset{\delta\rightarrow 0}{\overline\lim} \dfrac{|\nu |(B(\lambda , \delta ))}{\delta} > 0 \}) = 0.
 \end{eqnarray}
\end{lemma}

\begin{proof} Since $|\nu |\perp m$, we can find a Borel set $E_0$ such that $|\nu |(\mathbb{C}\setminus E_0 ) = 0$
and $m (E_0) = 0$. Let $N$ be a positive integer and let $E$ be a compact subset of $\{\lambda  :\underset{\delta\rightarrow 0}{\overline\lim} 
\dfrac{|\nu |(B(\lambda , \delta ))}{\delta} > \frac{1}{N} \} \setminus E_0$. Choose $\eta > 0$ and let $O$ be an open set 
containing $E$ such that $|\nu | (O) < \eta $. Now, for any point $x$ in $E$, there exists $ \delta _x > 0$ such that 
$|\nu |(B(x, \delta _x )) \ge \frac{1}{N} \delta _x$ and $B(x, \delta _x ) \subset O$. Since $E$ is a compact subset of 
$\cup_{x\in E} B(x, \delta _x)$, we can choose a finite subset $\{x_i\}_{i=1}^n$ of $E$ so that 
$E\subset \cup_{i=1}^n B(x_i, \delta _{x_i})$. From the $3r$-covering Theorem (see Theorem 2.1 in \cite{Tol14}), 
we can further select a subset $\{x_{i_j}\}_{j=1}^m$ such that 
$\{B(x_{i_j}, \delta _{x_{i_j}})\}_{j=1}^m$ are disjoint and 
 \[
 \ E \subset \cup_{i=1}^n B(x_i, \delta _{x_i}) \subset \cup _{j=1}^m  B(x_{i_j}, 3\delta _{x_{i_j}}).
 \] 
Therefore,
 \[
 \ \begin{aligned}
 \ m(E) \le & 10\sum _{j=1}^m \delta _{x_{i_j}} \\
 \ \le & 10N \sum _{j=1}^m |\nu |(B(x_{i_j}, \delta _{x_{i_j}})) \\ 
 \  = & 10N  |\nu |(\cup _{j=1}^mB(x_{i_j}, \delta _{x_{i_j}})) \\
 \ \le &10N  |\nu | (O) \\
 \ < & 10N \eta .
\ \end{aligned} 
 \] 
Since $\eta > 0$ is arbitrary, we can conclude that $m(E) = 0$, which establishes the result.
\end{proof} 
\bigskip

\begin{theorem}\label{GPTheorem}
(Plemelj's Formula for an arbitrary measure) Let $\nu$ be a finite, complex-valued Borel measure with 
compact support in $\mathbb{C}$. Suppose that $d\nu = hdm + d\sigma$ is the Radon-Nikodym 
decomposition with respect to $m$, where $h\in L^1(m)$ and $\sigma\perp m$. Then there exists a subset $Z$
of $\mathbb{T}$, where $m(Z) = 0$, such that the following hold. For any $\delta$, $0 < \delta < 1$, and any $\zeta$ in 
$\mathbb{T}\setminus Z$, there is a subset $E_{\delta}(\zeta)$ of $\overline{B(\zeta, \delta )}$ such that:

(a$\,'$) $\mathcal C(\nu ) (\zeta) = \lim_{\epsilon\rightarrow 0} \mathcal C_{\epsilon}(\nu)(\zeta)$ exists,

(b$\,'$) $\lim_{\delta\rightarrow 0}\frac{\gamma(E_{\delta}(\zeta ))}{\delta} = 0,$

(c$\,'$) $\mathcal C(\nu )(\lambda ) = \lim_{\epsilon\rightarrow 0}\mathcal{C}_ \epsilon (\nu) (\lambda )$ exists for $\lambda$ in
$B(\zeta, \delta )\setminus E_{\delta} (\zeta),$

(d$\,'$) 
\[
 \ \underset{\lambda \in S_r(\zeta)\setminus E_{\delta}(\zeta)}{\mathcal{C}(\nu )(\lambda )}  \stackrel{\delta}{\approx} 
\mathcal C(\nu ) (\zeta) + \frac{1}{2}h(\zeta) \bar{\zeta}, 
 \]
and

(e$\,'$)
\[
 \ \underset{\lambda \in T_r(\zeta)\setminus E_{\delta}(\zeta)}{\mathcal{C}(\nu )(\lambda )}  \stackrel{\delta}{\approx} 
\mathcal C(\nu ) (\zeta) - \frac{1}{2}h(\zeta) \bar{\zeta}.
 \]

 \end{theorem}

\begin{proof}
(Theorem \ref{GPTheorem}) From Tolsa's Theorem (4) and Lemma \ref{lemmaBasic}, we can find 
$Z\subset \partial \mathbb D$ with $m(Z) = 0$ such that, for $\zeta$ in $\mathbb{T}\setminus Z$,
 \[
 \ \underset{\delta\rightarrow 0}{\lim} \dfrac{|\sigma|(B(\zeta, \delta ))}{\delta} = 0,
 \]

 \[
 \ \mathcal{C}(\sigma)(\zeta) = \lim_{\epsilon\rightarrow 0} \mathcal{C}_{\epsilon}(\sigma)(\zeta)
 \]
exists and 
 \[
 \ \mathcal{C}(hm)(\zeta) = \lim_{\epsilon\rightarrow 0} \mathcal{C}_{\epsilon}(hm)(\zeta)
 \]
exists. So (a$\,'$) follows. Now, $\sigma$ satisfies the assumptions of Lemma \ref{CauchyTLemma} for $\zeta$
not in $Z$ and, therefore, by Lemma \ref{CauchyTLemma}, there exists $E_\delta(\zeta)$ such that 
$\lim_{\delta\rightarrow 0}\frac{\gamma(E_{\delta}(\zeta))}{\delta} = 0,$
\[
 \ \mathcal C(\sigma)(\lambda ) = \lim_{\epsilon\rightarrow 0} \mathcal{C}_ {\epsilon}(\sigma) (\lambda )
 \] 
exists for $\lambda$ in $B (\zeta, \delta ) \setminus E_{\delta}(\zeta)$, and 
\[
 \ \underset{\lambda \in B (\zeta, \delta ) \setminus E_{\delta}(\zeta)}{\mathcal{C}(\sigma) (\lambda )} 
\stackrel{\delta}{\approx} \mathcal C(\sigma) (\zeta). 
 \] 
Therefore, together with the classical Plemelj's Formula, we have: (b$\,'$), (c$\, '$), (d$\, '$) and (e$\,'$).
\end{proof}

The following Lemma is from Lemma B in \cite{ars02}. In its statement we let $m_2$ denote area measure
(i.e., two-dimensional Lebesgue measure) on the complex plane $\mathbb{C}$.

\begin{lemma} \label{lemmaARS}
There are absolute constants $\epsilon _1 > 0$ and $C_1 < \infty$ with the
following property. If $R > 0$ and $E \subset  \overline{B(0, R)}$ with 
$\gamma(E) < R\epsilon_1$, then
\[
\ |p(\lambda)| \le \dfrac{C_1}{\pi R^2} \int _{\overline{B(0, R)}\setminus E} |p|\, dm_2
\]
for all $\lambda$ in $B(0, \frac{R}{2})$ and all analytic polynomials $p$.
\end{lemma}

Our next lemma combined with Theorem \ref{GPTheorem} allows to give an alternate proof of 
Theorem III, part (a).

\begin{lemma}\label{NONTL}
Suppose that $\zeta\in\mathbb{T}$, that $0 < r < 1$ and that $f$ is an analytic function in $S_r(\zeta)$.
Furthermore, suppose that there exists $\delta_0$, $0 < \delta_0 < 1$, such that 
whenever $0 < \delta < \delta _0$, there is a subset $E_\delta (\zeta)$ of $B(\zeta, \delta )$ satisfying:

\[\lim_{\delta\rightarrow 0}\frac{\gamma(E_{\delta}(\zeta))}{\delta} = 0,\,\,\, \mbox{and}\,\,\,
 \ \underset{\lambda \in S_r(\zeta, \delta) \setminus E_{\delta}(\zeta)}{f(\lambda )} \stackrel{\delta}{\approx} b.
 \]
Then, for any $\rho$, $0 < \rho < r$
\[
\ \underset{\lambda \in S_{\rho}(\zeta)}{\lim _{\lambda\rightarrow\zeta}} f(\lambda ) = b.
 \]  
\end{lemma}

\begin{proof}
Suppose that $0 < \delta < \delta_0$ and that $\lambda \in S_{\rho}(\zeta, \delta/2)$. A simple geometry exercise shows that there 
exists $\epsilon_2 > 0$, that depends only on $\rho$, such that:
 \[
 \ B(\lambda,\epsilon_2 \delta) \subset S_r(\zeta, \delta).
 \]
Choose $0 < \delta_1 < \delta _0$ such that, for $0 < \delta < \delta _1$, we have 
\[\gamma(E_\delta(\zeta)) < \epsilon_1\epsilon_2\delta,\] 
where $\epsilon_1$ is from Lemma \ref{lemmaARS}. Using Lemma \ref{lemmaARS}, we get
 \[
 \ |f(\lambda ) - b | \le \dfrac{C_1}{\pi (\epsilon_2\delta)^2} \int _{B(\lambda,\epsilon_2 \delta)\setminus E_\delta(\zeta)}
|f(z) - b| dm_2(z) \le C_1 \sup _{z\in S_r(\zeta, \delta)\setminus E_{\delta}(\zeta)}|f(z) - b|.
 \]
Hence, the result follows.
\end{proof}

\begin{proof}
(Theorem III, part (a)) Suppose that $g\in L^s(\mu)$ ($\frac{1}{s} + \frac{1}{t} = 1$) such that $\int pg\, d\mu = 0$, for all the analytic polynomials $p$.
By the proof of \cite[Chapter VII, Lemma 1.7]{jbc}, we may assume that $|g| > 0$ a.e. $\mu$. Define finite, complex-valued 
Borel measure measures $\nu_1$ and $\nu_2$ by: $d\nu_1 = gd\mu$ and $d\nu_2 = fgd\mu$. We apply Theorem \ref{GPTheorem} to $\nu_1$ 
and $\nu_2$
separately and obtain sets $Z_1$ and $E_\delta^1(\zeta)$,  and $Z_2$ and $E_\delta^2(\zeta)$ for which the following hold. If $Z := Z_1\cup Z_2$ 
and $E_\delta(\zeta) := E_\delta^1(\zeta)\cup E_\delta^2(\zeta)$, then $m(Z) = 0$. And if $\zeta\in\mathbb{T}\setminus{Z}$, then, by Theorem IV, part (2):
 \[
 \ \lim_{\delta\rightarrow 0}\frac{\gamma(E_\delta(\zeta))}{\delta} \le A_T(\lim_{\delta\rightarrow 0}
\frac{\gamma(E_\delta ^1(\zeta))}{\delta} + \lim_{\delta\rightarrow 0}\frac{\gamma(E_\delta ^2(\zeta))}{\delta}) = 0.
 \] 
From (e$\,'$) in Theorem \ref{GPTheorem}, we see that, for $\zeta$ in $\mathbb{T}\setminus Z$,
 \begin{eqnarray} \label{ARSEq1}
 \ \ \mathcal C(g\mu ) (\zeta) = \frac{1}{2}g(\zeta) h(\zeta)\bar{\zeta},
 \end{eqnarray} 
and
 \begin{eqnarray} \label{ARSEq2}
 \ \ \mathcal C(fg\mu ) (\zeta) = \frac{1}{2}f(\zeta) g(\zeta) h(\zeta)\bar{\zeta}.
 \end{eqnarray}
For $\lambda$ in $\mathbb D$, we have:
\[
 \ f(\lambda )\mathcal C(g\mu ) (\lambda ) = \mathcal C(fg\mu )(\lambda ).
 \]
Applying (d$\,'$) in Theorem \ref{GPTheorem}, together with \eqref{ARSEq1} and \eqref{ARSEq2}, we conclude that
\[
 \ \underset{\lambda \in S_r(\zeta, \delta)\setminus E_{\delta}(\zeta)}{f(\lambda)g(\zeta)h(\zeta)\bar{\zeta}}  \stackrel{\delta}{\approx}  
f(\zeta) g(\zeta)h(\zeta)\bar{\zeta},
 \]
for $\zeta$ in $\mathbb{T}\setminus Z$.
Since $g$ is nonzero a.e. $\mu$ we can now apply Lemma \ref{NONTL} and get:
\[
 \ \underset{\lambda \in S_{\rho}}{\lim _{\lambda\rightarrow\zeta}} f(\lambda )  = f(\zeta),
 \]
whenever $0 < \rho < r$ and $\zeta\in\mathbb{T}\setminus Z$.  
It is well-known that the existence of nontangential limits is independent of $r$ up to sets of $m$-measure zero, and so our proof is complete.
\end{proof}

We observe that, using Theorem \ref{GPTheorem}, our proof of Theorem III, part (a) can be applied to a variety of (but not all) 
$R^t(K,\mu )$ spaces. 

\begin{remark}
Under the hypothesis of Theorem III, let $\mathcal{M}$ be a nontrivial, closed, shift invariant subspace of $P^t(\mu)$.
In the proof of \cite[Theorem~3.2]{ars02} (i.e., Theorem~III, part (b)) in order to establish that $\mbox{dim}(\mathcal{M}/z\mathcal{M}) =1$,
the authors show that if $\phi\in \mathcal{M}^{\perp}\subset L^s(\mu)$ and $f$ and $g$ are in $\mathcal{M}$, with $g(0) \ne 0$, then the 
function
\[\Phi(\lambda) := \int\dfrac{f(z) - \mbox{\tiny{$\dfrac{f(\lambda)}{g(\lambda)}$}}g(z)}{z - \lambda}\phi(z)\, d\mu(z) = 
\mathcal{C}(f\phi\mu)(\lambda) - \dfrac{f(\lambda)}{g(\lambda)}\mathcal{C}(g\phi\mu)(\lambda)\]
(which is meromorphic in $\mathbb{D}$) has zero nontangential boundary values on a set of 
positive $\mu|_{\mathbb{T}}$ measure, and hence is identically zero. We point out that this can 
also be accomplished by applying Theorem 3.6 in a proof that is similar to Theorem III, part (a),
above.
\end{remark}

We note that Theorem III, part (b) gives us an affirmative answer to the question:

\begin{question}
Let $\mu$ be supported in $\overline{\mathbb{D}}$ such that $P^2(\mu)$ is irreducible, 
$\mbox{abpe}(P^2(\mu))= \mathbb{D}$, and $\mu(\mathbb{T}) = 0$. Suppose that $f$ and $g$ 
are distinct, nontrivial functions in $P^2(\mu)$ that have nontangential limits on a set of positive $m$ 
measure, and let $\mathcal{M} = [f, g]$ -- the closure of $\{fp + gq: \mbox{$p$ and $q$ are polynomials}\}$ 
in $L^2(\mu)$. Does it follow that $\mathcal{M}\ominus z \mathcal{M}$ has dimension one?
\end{question}

We close with the question:

\begin{question}
What milder condition can be imposed, in place of the boundary values assumption on 
$f$ and $g$, and still get the conclusion (in Question 3.10) that $\mbox{dim}(\mathcal{M}\ominus z \mathcal{M}) = 1$?
\end{question}

\bibliographystyle{amsplain}

\end{document}